\documentclass[11pt, reqno]{amsart}
\usepackage{amsmath, amssymb, amsthm, graphicx,marvosym}
\usepackage{cancel}
\usepackage[nobysame]{amsrefs}[11pts]

\usepackage{color}

\usepackage{tikz}
\usepackage{caption}
\usepackage{amsmath}
\usepackage{amssymb}
\usepackage{amsthm}
\usepackage{enumerate}
\newtheorem{thm}{Theorem}
\newtheorem{lemma}{Lemma}[section]

\newtheorem{prop}[lemma]{Proposition}

\newtheorem{remark}[lemma]{Remark}
\newcommand \nc{\newcommand}
\newcommand{\ben}{\begin{eqnarray}}
\newcommand{\een}{\end{eqnarray}}
\newcommand{\beno}{\begin{eqnarray*}}
\newcommand{\eeno}{\end{eqnarray*}}

\makeatletter \@addtoreset{equation}{section} \makeatother

\nc{\ba}{\begin{array}}\nc{\ea}{\end{array}}

\nc{\be}{\begin{eqnarray}}\nc{\ee}{\end{eqnarray}}
\nc{\beq}{\begin{equation}}\nc{\eeq}{\end{equation}}
\nc{\bex}{\begin{eqnarray*}}\nc{\eex}{\end{eqnarray*}}
\nc{\btm}{\begin{theorem}} \nc{\etm}{\end{theorem}}
\nc{\blm}{\begin{lemma}} \nc{\elm}{\end{lemma}}
\nc{\va}{\varphi}
\nc{\ve}{\varepsilon}

\def\ds{\displaystyle}

 \newcommand{\n}{{\bf n}}

\newcommand{\R}{\mathbb {R}}

\def\({\left(\begin{array}{cccccc}}
\def\){\end{array}\right)}

\def\bes{\begin{eqnarray}}
\def\ees{\end{eqnarray}}

\let\th\theta
\let\epsilon\ve

\title[Wave model for liquid-crystals]{Singularity formation for the general Poiseuille flow of nematic liquid crystals }

\author{Geng Chen$^1$}
\address[G. Chen] {Department of Mathematics, University of Kansas, Lawrence, KS 66045, U.S.A.,}
\email{\tt gengchen@ku.edu}

\author{Majed Sofiani$^{2*}$}
\address[M. Sofiani]{Department of Mathematics, University of Kansas, Lawrence, KS 66045, U.S.A.}
\email{\tt sofiani@ku.edu}

\date{\today}

\begin{document}

\maketitle
$^1$ Department of Mathematics, University of Kansas, Lawrence, KS 66045, U.S.A.\\
Email: gengchen@ku.edu,\\
$^2$ Department of Mathematics, University of Kansas, Lawrence, KS 66045, U.S.A. Email: sofiani@ku.edu.

\begin{abstract}
We consider the Poiseuille flow of nematic liquid crystals via the full Ericksen-Leslie model. The model is described by a coupled system consisting of a heat equation and a quasilinear wave equation. In this paper, we will construct an example with a finite time cusp singularity due to the quasilinearity of the wave equation, extended from an earlier result on a special case.
\end{abstract}

{\em Dedicated to Professor Tong Zhang in the occasion of his 90-th birthday.}

\section{Introduction:}

The state of a nematic liquid crystal is  characterized by its velocity field ${\bf u}$ for the flow and its director field  ${\bf n}\in \mathbb S^2$ for the alignment of the rod-like feature.  These two characteristics interact with each other so that any distortion of the director $\n$ causes 
a motion ${\bf u}$  and, likewise,  any flow ${\bf u}$ affects the alignment $\n$.   
One famous model on nematic liquid crystal is the Ericksen-Leslie model for nematics that  was first proposed by Ericksen \cite{ericksen62} and Leslie \cite{leslie68}  in the 1960's. 

In this paper, we consider the Poiseulle flow via the full Ericksen-Leslie model, when
${\bf u}$ and ${\bf n}$ take the form
\[{\bf u}(x,t)= (0,0, u(x,t))^T \;\mbox{ and }\;{\bf n}(x,t)=(\sin\theta(x,t),0,\cos\theta(x,t))^T,\]
where $(x,t)\in {\mathbb R}\times {\mathbb R}^+$ and the motion ${\bf u}$ is along the $z$-axis and  the director ${\bf n}$ lies in the $(x,z)$-plane with angle $\theta$ made from the $z$-axis.  Then the Ericksen-Leslie system can be written as  
\begin{align}\label{sysf}
    \ds u_t&=\left(g(\theta)u_x+h(\theta)\theta_t\right)_x,\\\label{sysw}
\theta_{tt}+\gamma_1\theta_t&=c(\theta)\big(c(\theta)\theta_x\big)_x-h(\theta)u_x
\end{align}
where the $C^\infty$ functions $c,$ $g,$ and $h$ are explicitly given by  
\begin{align}\label{fgh}\begin{split}
 g(\theta):=&\alpha_1\sin^2\theta\cos^2\theta+\frac{\alpha_5-\alpha_2}{2}\sin^2\theta+\frac{\alpha_3+\alpha_6}{2}\cos^2\theta+\frac{\alpha_4}{2},\\
 %f(\theta)\,\equiv&c^2(\theta):=K_1\cos^2\theta+K_3\sin^2\theta,\\
  h(\theta):=&\alpha_3\cos^2\theta-\alpha_2\sin^2\theta=\frac{\gamma_1+\gamma_2\cos(2\theta) }{2},\\
  c^2(\theta):=&K_1\cos^2\theta+K_3\sin^2\theta.
 \end{split}
 \end{align}
Here, $K_1$, $K_3$ are positive elastic constants in the Oseen-Frank energy. 
The material coefficients $\gamma_1$ and $\gamma_2$ reflect the molecular shape and the slippery part between fluid and particles.
The coefficients $\alpha_i's$ and coefficients $\gamma_1$, $\gamma_2$ satisfy the following physical relations:
\begin{align}\label{a2g}
\gamma_1 =\alpha_3-\alpha_2,\quad \gamma_2 =\alpha_6 -\alpha_5,\quad \alpha_2+ \alpha_3 =\alpha_6-\alpha_5. 
\end{align}
The first two relations are compatibility conditions, while the third relation is called Parodi's relation, derived from Onsager reciprocal relations expressing the equality of certain relations between flows and forces in thermodynamic systems out of equilibrium (cf. \cite{Parodi70}). 
The coefficients also satisfy the following empirical relations (p.13, \cite{Les}) 
\begin{align}\label{alphas}
&\alpha_4>0,\quad 2\alpha_1+3\alpha_4+2\alpha_5+2\alpha_6>0,\quad \gamma_1=\alpha_3-\alpha_2>0,\\
&  2\alpha_4+\alpha_5+\alpha_6>0,\quad 4\gamma_1(2\alpha_4+\alpha_5+\alpha_6)>(\alpha_2+\alpha_3+\gamma_2)^2\notag.
\end{align} 
A detailed derivation of the Ericksen-Leslie system for Poiseulle flows can be found in \cite{CHL1}, where without loss of generality, we choose  density $\rho$ to be $1$, and the inertial coefficient $\nu$ of the director ${\bf n}$ to be $1$. Furthermore, for simplicity, the constant $a$ in the model in \cite{CHL1}, which is the gradient of pressure along the flow direction, is set to be zero.

%For the initial data, we assume  
%\beq\label{initial}
%u(x,0)=u_0(x)\in H^1(\R),\; \theta(x,0)=\theta_0(x)\in H^1\cap\,\mathcal{C}^{1,s}(\R),
%\; \theta_t(x,0)=\theta_1(x)\in L^2(\R),
%\eeq
%for some $s>0$. In addition, we impose the following technical assumption 
%\beq\label{tecinitial}
%g(\theta_0(x))u'_0(x)+h(\theta_0(x))\theta_1(x)\in L^1\cap L^\infty\cap C^\alpha(\R)
%\eeq
%for some $\alpha>0$.

%However, we do not need to mention these relations here since we are not assuming any specific values for $\{\alpha_i\}_{i=1}^6.$ 

\medskip

Due to the quasilinearity of the wave equation on $\theta$,  finite time gradient blowup might happen even when the initial data are smooth. In \cite{CHL1}, Chen, Huang and Liu established an example showing such kind of singularity formation phenomena, for a special case of \eqref{sysf}-\eqref{sysw} when $g=h=1$ and $\gamma_1=2$. 

The construction of the singularity formation example relies on the framework in \cite{GHZ} by Glassy-Hunter-Zheng on the variational wave equation
\beq\label{vwth}  \theta_{tt}=c(\theta)(c(\theta)\theta_x)_x.
\eeq
The global well-posedness theories of H\"older continuous solutions for variational wave equations and systems related to nematic liquid crystals have been intensively studied in the last two decades \cites{BC2015,BC,BCZ,BH,BZ,CCD,CZZ12,CZ12,HR,ZZ03,ZZ10,ZZ11
}.

The proof of finite time singularity formation in \cite{CHL1} for the special case of \eqref{sysf}-\eqref{sysw} also relies on a crucial estimate on the $L^\infty$ bound of a new function $J$, where in the general case, 
\beq\label{Jdef}
J=u_x+\frac{h}{g} \theta_t.
\eeq
We note that functions $u_x$ and $\theta_t$ both blow up when singularity forms, however, their combination $J$ will be proved uniformly bounded. This estimate, first obtained in \cite{CHL1} for the special case (when $g=h=1$ and $\gamma_1=2$), is crucial for both singularity formation and global existence of H\"older continuous solution for the system \eqref{sysf}-\eqref{sysw}. 

Using \eqref{Jdef},  the wave equation \eqref{sysw} can be written as
\begin{align}\label{waveAx}
    \theta_{tt}+\big[\gamma_1-\frac{h^2(\theta)}{g(\theta)}\big]\theta_t&=c(\theta)(c(\theta)\theta_x)_x-h(\theta)J.
\end{align}
From \eqref{fgh}-\eqref{alphas}, we have the following bounds for $g$, $h$ and $c$
\begin{align}\nonumber
    g_L&\leq g(\th)\leq g_U,\\\label{ghcbounds}
    h_L&\leq h(\th)\leq h_U,\\\nonumber
    C_L&\leq c(\th)\leq C_U
\end{align}
where $g_L,g_U,h_L,h_U,C_L$ and $C_U$ are constants such that $g_L,g_U,h_U, C_L$ and $C_U$ are strictly positive.
Furthermore,  physical laws in \eqref{a2g} and \eqref{alphas} give that 
\beq\label{crucial}
\min\{\gamma_1-\frac{h^2(\theta)}{g(\theta)},g(\theta)\}>\overline C.
\eeq
for some positive constant $\overline C$, where the proof can be found in \cite{CHL1}.

In this paper, we will first find a bound on $J$ in terms of the initial energy. So \eqref{waveAx} is  a damped variational wave equation adding a uniform bounded source term $hJ$. Then we can prove  the singularity formation of cusp singularity using the methods in \cites{GHZ, CZ12,CHL1}. Note the source term $h u_x$ in the original equation \eqref{sysw} may be unbounded.

%Then the singularity formation example for the system \eqref{sysf}-\eqref{sysw} can be constructed using the method in \cite{GHZ} for \eqref{vwth} and \cite{CZ12} for  \eqref{vwth} with a damping term. 

In another companion paper \cite{CLM}, we will show that $J$ is bounded under the norm of $L^2\cap L^\infty\cap C^\alpha$ for some positive constant $\alpha$, even for weak solutions including singularities. This is one of the key estimates for the global existence proof. 

Except showing one example forming the cusp singularity, another motivation of this paper is to introduce the major mathematical idea why we can get better regularity on $J$ than $u_x$ or $\theta_t$, in the general case. For smooth solutions, we can explain the idea in a relatively easier manner than for weak solutions.

Now, let's expain the main difficulty in controlling $J$ for the general case comparing to the special case. For the special case considered in \cite{CHL1}  (when $g=h=1$ and $\gamma_1=2$), \eqref{sysf} is 
\[
u_t=u_{xx}+\theta_{tx},
\]
so $u$ can be solved directly by $\theta_{tx}$.
However, this is not true for the general case.
The main difficulty we need to overcome is the varying coefficient $g(\theta)$ in the heat equation \eqref{sysf}. Although $g(\theta)$ is strictly positive and uniformly bounded, it is only H\"older continuous on $x$ and $t$ at the blowup. Therefore, the derivatives of $g(\theta)$ on $x$ and $t$ blow up when singularity forms. This creates a lot of essential problems for finding a uniform bound on $J$, since we need to use both heat and wave equations to bound $J$. One of our key ideas is to consider the potential $A$ of $J$, with $A_x=J$.

Before we state the main result we define the following function  $\phi(x)$ with $x\in \mathbb{R}$, that is used to design the initial data. Take $\phi\in C^1$ with
\begin{align}
    \phi(0)=0\,\,\,\,\,\,\,\,\text{and}\,\,\,\,\,\,\, \phi(a)=0\,\,\,\text{for}\,\,\,\, a\notin(-1,1),
\end{align}
\begin{align}
-\phi'(0)>\max\{\frac{16C_U\|\gamma_1-\frac{h^2}{g}\|_{L^\infty}}{ c'(\theta*)C_L\ln 2},\frac{\exp\big({\|{\gamma_1-\frac{h^2}{g}\|_{L^\infty}}}\big)}{C_L}\}\,\,\,\text{and}\,\,\,\,|\phi'(x)|\leq C_2,
\end{align}
and
\begin{equation}
\int_{-1}^{1}(\phi')^2(a)\,da<k_0,
\end{equation}
where $\theta^*$ is a constant such that $c'(\theta^*)>0$ and $C_2, k_0$ are some positive constants.\\

Here is the main singularity formation result. 

\begin{thm}\label{thm}
Let the initial data be
\begin{align}
\theta(x,0)=\theta_0(x):=\theta^*+\epsilon\phi(\frac{x}{\epsilon}),\,\,\,\,\, \theta_t(x,0)=\theta_1(x):=(-c(\theta_0(x))+\epsilon)\theta'_0(x),
\end{align}
and
\[u(x,0)=u_0(x):= \begin{cases} 
      0, & x\in(-\infty,-\epsilon) \\
      \int_{-\epsilon}^{x}\frac{h}{g}c(\theta_0(a))\theta'_0(a)\,da, &  x\in[-\epsilon,\epsilon] \\
      \chi(x), & x\in(\epsilon,\epsilon+2)\\
      0, & x\in(\epsilon+2,\infty) 
   \end{cases}
\]
where $\theta^*$ is a constant such that $c'(\theta^*)>0$, and $\chi(x)$ is a $C^1$ function satisfying
\begin{equation}
|\chi'(x)|\leq \frac{3}{2}\|\frac{h}{g}\|_{L^\infty} C_UC_2\,\ve,
\end{equation}
and such that $u_0(x)$ is a $C^1$ function.
 Then there exists a sufficiently small positive choice of the parameter $\epsilon$ such that the solution $(\th,u)$  of \eqref{sysf}-\eqref{sysw} is $C^1$ only up to a finite time. More precisely, the solution is continuously differentiable up to some time before \[T=\min\left\{\frac{2\ln2}{\|\gamma_1-\frac{h^2}{g}\|_{L^\infty}},1\right\}.\]
\end{thm}

In the proof of the theorem we will show that the singularity happens in the following form: there exists a time $0<t_0<T$ such that
\begin{align}
   \th_t\to \infty,\qquad&\th_x\to -\infty 
\end{align}
as $t\to t_0^-$ along the characteristic having the blowup phenomena.
%
%
%Introducing $v=\int_{-\infty}^x u\,dx$ and using \eqref{sysf} we obtain
%\beq\label{veq}
%v_t=gv_{xx}+h \theta_t=g u_x+h \theta_t.
%\eeq
%Define
%\[
%J=\frac{v_t}{g}=u_x+\frac{h}{g}\theta_t.
%\]
%
%
%Using the equation of $v$, it is easy to check that $J$ satisfies:
%\beq\label{J_eq}
%(gJ)_t=g (gJ)_{xx}+h \theta_{tt}+h' \theta_t^2+g'(\theta)\theta_t
%u_{x}\eeq
In this example, the initial energy is small, and initial data are smooth, but finite time singularity still forms. The singularity is a cusp type of singularity, combining with the existence result in \cite{CLM}, also see \cites{CHL1,GHZ}.

The rest of the paper is divided into 3 sections as follows. In Section 2 we show the decay of the energy associated with a smooth solution. Section 3 contains the main estimate on the quantity $J$: we show the uniform bound of $J$ over a fixed time interval. In Section 4, we prove the singularity formation result in Theorem \ref{thm}.

\section{The energy of the system}
The energy of the system \eqref{sysf}-\eqref{sysw} is defined as
\begin{align}\label{Edef}
    \mathcal{E}(t):=\frac{1}{2}\int_\R\theta_t^2+c(\theta)^2\theta_x^2+u^2\,dx.
\end{align}
In this section we show the energy decay for smooth solutions.
\begin{prop}
For any smooth solution $(\theta(x,t),u(x,t))$ of the system \eqref{sysf}, \eqref{sysw}, the energy $\mathcal{E}(t)$ decays with the following rate
\begin{align}\label{engineq}
\begin{split}
\frac{d}{dt}\mathcal{E}(t)= -\int_{\R}\Big(b(\theta)u_x^2+\gamma_1\Big( \theta_t+\frac{h(\theta)}{\gamma_1}u_x\Big)^2\Big)\;dx,
\end{split}
\end{align}
where  $b(\theta):=g(\theta)-\frac{h^2(\theta)}{\gamma_1}>0.$
\end{prop}

\begin{proof}
Multiplying \eqref{sysf} by $u$ and \eqref{sysw} by $\theta_t$, and integrating by parts, we have 
\beq\label{en1}
\frac{1}{2}\frac{d}{dt}\int u^2\,dx=-\int g(\theta)u_x^2\,dx-\int h(\theta)\theta_tu_x \,dx,
\eeq
and 
 \begin{align}\label{en2}\begin{split}
\frac{1}{2}\frac{d}{dt}\int \left(\theta_t^2+c^2(\theta)\theta_x^2\right)\,dx=&-\int \gamma_1\theta_t^2\,dx
-\int h(\theta)u_x\theta_t\,dx.
\end{split}
\end{align}
Sum up \eqref{en1} and \eqref{en2} to get
\begin{align*}
\frac{1}{2}\frac{d}{dt}\int \left(\theta_t^2+ c^2(\theta)\theta_x^2+ u^2\right)\,dx=&-\int \left(\gamma_1\theta_t^2+2h(\theta)\theta_tu_x+g(\theta)u_x^2\right) \,dx\\
=& -\int_{\R}\Big(b(\theta)u_x^2+\gamma_1\big(\theta_t+\frac{1}{\gamma_1}h(\theta)u_x\big)^2\Big)\,dx.
\end{align*}
\end{proof}

Therefore, considering only smooth solutions, for any $T>0$, we have 
\begin{align}\label{Ene}
    \max_{0\leq t\leq T}\mathcal{E}(t)\leq \mathcal{E}(0).
\end{align}

\section{Uniform bound on $J$ in finite time}
For convenience, we just fix $T_0=1$ and give a uniform upper bound on $J$, defined in \eqref{Jdef}, when $0\leq t\leq T_0$, when solution is smooth. In fact, one can choose $T_0$ to be any positive constant, and still prove that $J$ is bounded.

Before our calculation, we recall the bounds in \eqref{ghcbounds} and \eqref{crucial}, and also note that $|h'|$ and $|g'|$ are bounded above by a constant because of \eqref{fgh}.
\subsection{An integral relation on $J$ and its estimate}

To get the uniform $L^\infty$ estimate on $J$, we first find an integral relation on $J$ using both \eqref{sysf} and \eqref{sysw}. 

We study the potential of $J$, defined as 
\[A:=\int_{-\infty}^x u_x+\frac{h}{g}\theta_t\,dz=\int_{-\infty}^x J\,dx.\] 
So
\[
A_x=J.
\]

In terms of the quantity $J$ we can write the system as
\begin{align}
    u_t&=\big(g(\theta)u_x+h(\theta)\theta_t\big)_x,\\
    \th_{tt}+(\gamma_1-\frac{h^2}{g})\th_t&=c(\th)\big(c(\th)\th_x\big)_x-hJ.
\end{align}
And we derive the following equation for $A$.
\begin{align*}
    A_t&={\int_{-\infty}^x\frac{(gJ)_t}{g}\,dz}-{\int_{-\infty}^x\frac{g'\theta_t}{g}J\,dz}\\
    &={(gJ)_x+\int_{-\infty}^x\frac{1}{g}\bigg[[g'\theta_t-\gamma_1 g]J+[h'-\frac{g'h}{g}]\theta_t^2+[\gamma_1g-h^2]u_x+h\,c(c(\theta)\theta_x)_x\bigg]\,dz}\\
    &{-\int_{-\infty}^x\frac{g'\theta_t}{g}J\,dz}\\
    &=\big(g(\theta)J\big)_x+\int_{-\infty}^x\frac{1}{g}\bigg[[-\gamma_1]gJ+[h'-\frac{g'h}{g}]\theta_t^2+[\gamma_1g-h^2]u_x+h\,c(c(\th)\theta_x)_x\bigg]\,dz.
\end{align*} 
Integrating it by parts, we get
\begin{align}\nonumber
    A_t&=g(\theta)A_{xx}-\gamma_1A+g'\theta_xJ\\\nonumber
    &+\bigg[\int_{-\infty}^x[\frac{h'}{g}-\frac{g'h}{g^2}]\theta_t^2-[\gamma_1-\frac{h^2}{g}]'\theta_zu-(\frac{h(\theta)c(\theta)}{g})'c(\theta)\theta_z^2\,dz\bigg]\\
    &\quad+[\gamma_1-\frac{h^2}{g}]u+\frac{h(\theta)c^2(\theta)}{g}\theta_x.
\end{align}

In summary we consider the following Cauchy problem
\begin{align}\label{AAeq}
    A_t-g(\theta)A_{xx}+\gamma_1A= g'\theta_x J+F_1(\theta,u)+F_2(\theta,u),
\end{align}
where 
\begin{align*}
F_1&=\int_{-\infty}^x\left\{[\frac{h'}{g}-\frac{g'h}{g^2}]\theta_t^2-[\gamma_1-\frac{h^2}{g}]'\theta_z u-(\frac{h(\theta)c(\theta)}{g})'c(\theta)\theta_z^2\right\}\,dz,\\
   F_2 &=[\gamma_1-\frac{h^2}{g}] u+\frac{h(\theta)c^2(\theta)}{g}\theta_x,  
\end{align*}
with
\begin{align}
    A(x,0)=\int_\R u_0'+\frac{h(\theta_0)}{g(\theta_0)}\theta_1\,dx:=A_0(x).
\end{align}

It is easy to verify that, there exists some positive constant $C$, such that
\beq\label{Fest}
\|F_1(\theta,u)\|_{L^\infty(\mathbb R)}(t)\leq C\mathcal E(t)\leq C\mathcal E(0),\quad 
\|F_2(\theta,u)\|^2_{L^2(\mathbb R)}(t)\leq C\mathcal E(t)\leq C\mathcal E(0).
\eeq
In this section, without confusion, we always use $C$ to denote different positive constants for different estimates.

Now using the conclusion from Chapter 1, Theorem 12 in \cite{Fri}, we can formally solve $A$ by \eqref{AAeq} as
\begin{align}\label{Asol0}
    A(x,t)=&\int_\R\Gamma(x,t;\xi,0)A_0(\xi)\,d\xi\nonumber\\
    &+\int_0^t\int_\R\Gamma(x,t;\xi,\tau)\big(g'\theta_x J+F_1(\theta,u)+F_2(\theta,u)\big)(\xi,\tau)\,d\xi\,d\tau,
\end{align}
where the kernel $\Gamma$  can be written in terms of the heat kernel as follows,
 \begin{align}
     \Gamma(x,t,\xi,\tau)&=H^{\xi,\tau}(x-\xi,t-\tau)
+\int_\tau^t\int_\mathbb{R}H^{y,s}(x-y,t-s)\Phi(y,s;\xi,\tau)\,dy\,ds,
 \end{align}
where
\begin{align}
    H^{\xi,\tau}(x-\xi,t-\tau)=\frac{\sqrt{g(\theta(\xi,\tau))}}{2\sqrt{\pi}\sqrt{t-\tau}}e^{-\frac{g(\theta(\xi,\tau))(x-\xi)^2}{4(t-\tau)}}.
\end{align}
The function $\Phi$ is determined by the condition
\[\mathcal{L}\,\Gamma=0,\] 
where \[\mathcal{L}:=\partial_t-g(\theta)\partial_{xx}+\gamma_1.\]

It can be shown that such function exists and
\begin{align}\label{phiest}
|\Phi(y,s;\xi,\tau)|\leq \frac{C}{(s-\tau)^{5/4}}e^{\frac{-d(y-\xi)^2}{4(s-\tau)}},
\end{align}
where $d$ is a constant depending on $g.$ Moreover, we have the following bounds for $\Gamma$ and $\Gamma_x,$

\begin{align}\label{Gammaest}
    |\Gamma(x,t;\xi,\tau)|\leq \frac{C}{\sqrt{t-\tau}}e^{-\frac{d(x-\xi)^2}{4(t-\tau)}},
\end{align}
\begin{align}\label{Gamma_xest}
    |\Gamma_x(x,t;\xi,\tau)|\leq \frac{C}{{t-\tau}}e^{-\frac{d(x-\xi)^2}{4(t-\tau)}}.
\end{align}
The reader can find the proof of the above estimates in Chapter 1, Theorem 11 in \cite{Fri}.\\

Now, differentiating \eqref{Asol0} w.r.t $x$ we obtain that $J$ satisfies the following relation 
\begin{align}\label{Asol}
    J(x,t)&=\int_\R\Gamma_x(x,t;\xi,0)A_0(\xi)\,d\xi\nonumber\\
      &  +\int_0^t\int_\R\Gamma_x(x,t;\xi,\tau)F_1(\theta,u)(\xi,\tau)\,d\xi\,d\tau\nonumber\\
&        \int_0^t\int_\R\Gamma_x(x,t;\xi,\tau)F_2(\theta,u)(\xi,\tau)\,d\xi\,d\tau\nonumber\\
    &+\int_0^t\int_\R\Gamma_x(x,t;\xi,\tau)(g'\theta_\xi J)(\xi,\tau)\,d\xi\,d\tau\nonumber\\
    &:=L_0(x,t)+L_1(x,t)+L_2(x,t)+\int_0^t\int_\R\Gamma_x(x,t;\xi,\tau)(g'\theta_\xi J)(\xi,\tau)\,d\xi\,d\tau.
\end{align}

Then we use this expression to find the uniform upper bound on $J$. First, we give uniform bounds on $L_0$, $L_1$ and $L_2$ in terms of $J_0$, $A_0$ and the initial energy.  
 %In addition to the term from the initial data
%\begin{align*}
    %L_0(x,t):=\int_{R}\Gamma_x(x,t;\xi,0)A_0(\xi)\,d\xi.
%\end{align*}

\subsection{$L^\infty$ estimates on $L_1$ and $L_2$} %$\mathcal{M}(J)-gJ^0(x,t)$\\}

%To show $M_x\in L^\infty$: We start with the term $$G:=\int_{-\infty}^x[\frac{h'}{g}-\frac{g'h}{g^2}]\theta_t^2-[\gamma_1-\frac{h^2}{g}]'\theta_zu-(\frac{h(\theta)c(\theta)}{g})'c(\theta)\theta_z^2\,dz$$

Recall the estimates \eqref{Gamma_xest} and \eqref{Fest}. First,
\beq\label{L1inf}|L_1|\leq \|F_1\|_{L^\infty}\int_0^t\int_R\frac{C}{{t-\tau}}e^{-\frac{d(x-\xi)^2}{4(t-\tau)}}\,d\xi\,d\tau\leq Ct^{1/2}\|F_1\|_{L^\infty(\Omega_t)}\leq t^{\frac{1}{2}}C\mathcal E(0).\eeq
Secondly,
\begin{align}
 |L_2|\leq& \bigg[\int_0^t\int_R\frac{1}{{|t-\tau|}^{2-2r}}e^{-\frac{d(x-\xi)^2}{2(t-\tau)}}\,d\xi\,d\tau\bigg]^{1/2} \bigg[\int_0^t\int_R\frac{1}{{|t-\tau|}^{2r}}F_2^2\,d\xi\,d\tau\bigg]^{1/2}\nonumber\\
 &\leq \bigg[\int_0^t\frac{1}{{|t-\tau|}^{\frac{3}{2}-2r}}\,d\tau\bigg]^{1/2}\bigg[\int_0^t\frac{1}{{|t-\tau|}^{2r}}\,d\tau\bigg]^{1/2}\|F_2\|_{L^\infty((0,t),L^2(R))}.\label{L2inf0}
\end{align}
For $r=\frac{3}{8}$ we obtain:
\beq\label{L2inf}|L_2|\leq t^{\frac{1}{4}}C\|F^2\|_{L^\infty((0,t),L^2(\R))}
 \leq t^{1/4}C\mathcal E^\frac{1}{2}(0).\eeq

\subsection{$L^\infty$ estimates on $L_0$} 

\begin{lemma}\label{lemmainit}
Let $A_0(x)$ be such that $A_0(x),A_{0,x}(x)\in L^\infty(\R),$ we have
\begin{align}
    &\bigg|\int_\R\Gamma_x(x,t;\xi,0)A_0(\xi)\,d\xi\bigg|\\
    \leq& C\|J_{0}(x)\|_{L^\infty(\R)}+C_1\|\theta_0'(x)\|_{L^\infty(\R)}\|A_0(x)\|_{L^\infty(\R)}+C_2\|A_0(x)\|_{L^\infty(\R)},
\end{align}
for some constants $C,C_1$ and $C_2$.
\end{lemma}

%The idea of the proof, briefly, is the following: we write $H_x= H_\xi+\dots$ then we integrate the first term by parts and control the extra terms.\\
\begin{proof}
Recall: 
\begin{align}\label{Gamma_x}
\Gamma_x(x,t;\xi,0)=H_x^{\xi,0}(x-\xi,t)+\int_0^t\int_\mathbb{R}H_x^{y,s}(x-y,t-s)\Phi(y,s;\xi,0)\,dy\,ds
\end{align}
and
\begin{align}
    H^{\xi,\tau}(x-\xi,t-\tau)=\frac{\sqrt{g(\theta(\xi,\tau))}}{2\sqrt{\pi}\sqrt{t-\tau}}e^{-\frac{g(\theta(\xi,\tau))(x-\xi)^2}{4(t-\tau)}}.
\end{align}
Also note that due to the dependence of $H^{\xi,0}(x-\xi,t)$ on $g(\th(\xi,0))$ we have the following relation
\begin{align}\nonumber
    H_x^{\xi,0}(x-\xi,t)&=-H_\xi^{\xi,0}(x-\xi,t)+\frac{g'\theta_0'}{4\sqrt{\pi}\sqrt{g}\sqrt{t}}e^{-g(\theta_0(\xi)(x-\xi)^2/4t}\\\nonumber
    &-\frac{\sqrt{g}g'\theta_0'(x-\xi)^2}{\sqrt{\pi}8t^{3/2}}e^{-g(\theta_0(\xi)(x-\xi)^2/4t}\\\nonumber
    &=-H_\xi^{\xi,0}(x-\xi,t)+\frac{g'\theta_0'}{2\sqrt{\pi}}H^{\xi,0}(x-\xi,t)\\\label{init}
    &-\frac{\sqrt{g}g'\theta_0'(x-\xi)^2}{\sqrt{\pi}8t^{3/2}}e^{-g(\theta_0(\xi)(x-\xi)^2/4t}.
\end{align}

Starting with the first term of $\Gamma_x$ and using the above relation, there are three integrals to estimate. The integral related to the first term in \eqref{init} is
\begin{align}
\int_\R -H_\xi^{\xi,0}(x-\xi,t)A_0(\xi)\,d\xi=\int_\R H^{\xi,0}(x-\xi,t)A_{0,\xi}(\xi)\,d\xi.
\end{align}
We have
\begin{align}\nonumber
\bigg|\int_\R -H_\xi^{\xi,0}(x-\xi,t)A_0(\xi)\,d\xi\bigg|&\leq\int_\R H^{\xi,0}(x-\xi,t)|A_{0,\xi}(\xi)|\,d\xi\\\nonumber
&\leq C \|A_{0,x}(x)\|_{L^\infty(\R)}\\
&=C\|J_{0}(x)\|_{L^\infty(\R)}.
\end{align}
The integral related to the second term of \eqref{init} can be easily estimated as
\begin{align}
   \big|\int_\R \frac{g'\theta_0'}{2\sqrt{\pi}}H^{\xi,0}(x-\xi,t) A_0(\xi)\,d\xi\big|\leq C_1\|\th'_0(x)\|_{L^\infty(\R)} \|A_0(x)\|_{L^\infty(\R)}.
\end{align}
The integral related to the third term of \eqref{init} has a similar bound as the second term, using the change of variable $$u=\frac{x-\xi}{\sqrt{t}}.$$
Combing above estimates, we obtain the following estimate 
\begin{align}\label{3.25}
    \bigg|\int_\R H_x^{\xi,0}(x,t;\xi,0)A_0(\xi)\,d\xi\bigg|\leq C \|J_{0}(x)\|_{L^\infty(\R)}+C_1\|\theta_0'(x)\|_{L^\infty(\R)}\|A_0(x)\|_{L^\infty(\R)}.
\end{align}
For the term with $\Phi$ (the second term in \eqref{Gamma_x}) we have:
\begin{align*}
&\bigg|\int_\R\int_0^t\int_\mathbb{R}H_x^{y,s}(x-y,t-s)\Phi(y,s;\xi,0)\,dy\,ds\,A_0(\xi)\,dy\,ds\,d\xi\bigg|\\
\leq& \|A_0(x)\|_{L^\infty}\int_\R\int_0^t\int_\R \frac{|x-y|}{(t-s)^{3/2}}e^{-g_L\frac{(x-y)^2}{4(t-s)}} \frac{1}{s^{5/4}}e^{-d\frac{(y-\xi)^2}{4s}}\,dy\,ds\,d\xi\\
\leq& C\,\|A_0(x)\|_{L^\infty}\int_0^t\int_\R \frac{|x-y|}{(t-s)^{3/2}}e^{-g_L\frac{(x-y)^2}{4(t-s)}} \frac{1}{s^{3/4}}\,dy\,ds\\
\leq &C\,\|A_0(x)\|_{L^\infty}\int_0^t \frac{1}{(t-s)^{1/2}} \frac{1}{s^{3/4}}\,ds\leq C_2 \|A_0(x)\|_{L^\infty(\R)}.
\end{align*}
Combining this estimate and \eqref{3.25}, we prove the lemma.
\end{proof}

\subsection{Uniform bound on $J$}

 Similarly as in \eqref{L2inf0}, we have 
\begin{align}
&| \int_0^t\int_\R\Gamma_x(x,t;\xi,\tau)(g'\theta_x J)(\xi,\tau)\,d\xi\,d\tau|\nonumber\\
\leq &
C t^{1/4}\|J\|_{L^\infty((0,t),L^\infty(\R))}\|\theta_x\|_{L^\infty((0,t),L^2(\R))}]\nonumber\\
 \leq &C t^{1/4}\|J\|_{L^\infty((0,t),L^\infty(\R))}\mathcal E^\frac{1}{2}(0),\label{E0J}
\end{align}
where we use the fact that $g'$ is uniformly bounded, \eqref{Edef}, \eqref{ghcbounds} and \eqref{Ene}.

Suppose that 
\beq\label{ce12}
C{\mathcal E}^\frac{1}{2}(0)\leq\frac{1}{2}
\eeq
 for the constant $C$ in \eqref{E0J}, then we have 
\beq\label{L3inf}
| \int_0^t\int_\R\Gamma_x(x,t;\xi,\tau)(g'\theta_x J)(\xi,\tau)\,d\xi\,d\tau|
\leq \frac{1}{2}\|J\|_{L^\infty((0,1),L^\infty(\R))},
\eeq
for any $t\in[0,T_0]$ with $T_0=1$.

Hence by \eqref{Asol}, \eqref{L1inf}, \eqref{L2inf}, \eqref{L3inf} and Lemma \ref{lemmainit}, we have 
\begin{align}\nonumber
\|J\|_{L^\infty((0,1),L^\infty(\R))}&\leq 2\big[C\mathcal E(0)+C\mathcal E^{\frac{1}{2}}(0)+C\|J_{0}(x)\|_{L^\infty(\R)} \\
&\ \ \ +C_1\|\theta_0'(x)\|_{L^\infty(\R)}\|A_0(x)\|_{L^\infty(\R)}+C_2\|A_0(x)\|_{L^\infty(\R)}\big].
\label{key0}
\end{align}

\section{Singularity formation for classical solutions:}
In this section we prove Theorem \ref{thm}. 
%Since we already obtained a similar bound on $J$ as in the special case considered in \cite{CHL1},  the singularity formation can be proved in a similar way as in \cite{CHL1}.
For reader's convenience, we recall the system
\begin{align}\label{u}
u_t&=\big(g(\theta)u_x+h(\theta)\theta_t\big)_x,\\ \label{Theta}
\theta_{tt}+\gamma_1\theta_t&=c(\theta)(c(\theta)\theta_x)_x-h(\theta)u_x,
\end{align}
%Define \[v(t,x)=\int_{-\infty}^x u(t,z)\,dz.\]In terms of $v$ we can write:
and the $C^1$ initial data
\begin{align}
\theta(x,0)=\theta_0(x)&=\theta^*+\epsilon\phi(\frac{x}{\epsilon}),\\ 
\theta_t(x,0)=\theta_1(x)&
=(-c(\theta_0(x))+\epsilon)\phi'(\frac{x}{\epsilon}),
\end{align}
\[u(x,0)=u_0(x)= \begin{cases} 
      0, & x\in(-\infty,-\epsilon) \\
      \int_{-\epsilon}^{x}\frac{h}{g}c(\theta_0(a))\theta'_0(a)\,da, &  x\in[-\epsilon,\epsilon] \\
      \chi(x), & x\in(\epsilon,\epsilon+2)\\
      0, & x\in(\epsilon+2,\infty)
   \end{cases}
\]
where $\chi(x)$ is a $C^1$ function satisfying
\begin{equation}
|\chi'(x)|\leq \frac{3}{2}\|\frac{h}{g}\|_{L^\infty} C_UC_2\,\epsilon,
\end{equation}
and the $C^1$ function $\phi$ satisfies 
\begin{align}
    \phi(0)=0\,\,\,\,\,\,\,\,\text{and}\,\,\,\,\,\,\, \phi(a)=0\,\,\,\text{for}\,\,\,\, a\notin(-1,1),
\end{align}
\begin{align}\label{max}
-\phi'(0)>\max\{\frac{16C_U\|\gamma_1-\frac{h^2}{g}\|_{L^\infty}}{ c'(\theta^*)C_L\ln 2},\frac{\exp\big({\|{\gamma_1-\frac{h^2}{g}\|_{L^\infty}}}\big)}{C_L}\}\,\,\,\text{and}\,\,\,\,|\phi'(x)|\leq C_2,
\end{align}
and
\begin{equation}
\int_{-1}^{1}(\phi')^2(a)\,da<k_0.
\end{equation}
Here $\theta^*$ is a constant such that $c'(\theta^*)>0$ and $k_0$ is some constant. 

\begin{remark}
A choice of the function $\chi(x)$ can be a cubic polynomial constructed by satisfying the following constraints,
\begin{align*}
&\chi(\epsilon)=\int_{-\epsilon}^{\epsilon}\frac{h}{g}c(\theta_0(a))\theta_0'(a)\,da,\\
&\chi'(\epsilon)=\frac{h}{g}c(\theta_0(\epsilon))\theta_0'(\epsilon)=\frac{h}{g}c(\theta_0(\epsilon))\phi'(1)=0,\\
&\chi(\epsilon+2)=0,\\
&\chi'(\epsilon+2)=0.
\end{align*}
Some calculations lead to 
\begin{align*}
\chi(x)&=\big(\frac{1}{4}\int_{-\epsilon}^\epsilon \frac{h}{g} c\phi'\,da\big)x^3-\big(\frac{3}{4}(\epsilon+1)\int_{-\epsilon}^\epsilon \frac{h}{g}c\phi'\,da\big)x^2+\big(\frac{3}{4}\epsilon(\epsilon+2)\int_{-\epsilon}^\epsilon\frac{h}{g} c\phi'\,da\big)x\\
&-\big(\frac{1}{4}(\epsilon-1)(\epsilon+2)^2\int_{-\epsilon}^\epsilon \frac{h}{g} c\phi'\,da\big).
\end{align*}
The derivative $\chi'(x)$ is a quadratic polynomial that vanishes at the end points $x=\epsilon$ and $x=\epsilon+2.$ The parabola is either concave or convex depending on the sign of the integral $\int_{-\epsilon}^\epsilon c\ \phi'\,da.$ In any case,
the critical point of the derivative is $$\big(\epsilon+1,\chi'(\epsilon+1)\big)=\big(\epsilon+1,-\frac{3}{4}\int_{-\epsilon}^\epsilon\frac{h}{g} c\ \phi'\,da\big).$$ This means
\[|\chi'(x)|\leq\frac{3}{2}\|\frac{h}{g}\|_{L^\infty} C_UC_2\,\ve.\]
\end{remark}

Now we define the following gradient variables representing the rate of change of $\th$ along the forward and backward characteristics:
$$S:=\theta_t-c(\theta)\theta_x,$$ $$R:=\theta_t+ c(\theta)\theta_{x}.$$
Direct calculations give:
\begin{align}\label{Sf}
S_t+c(\th)S_x=\frac{c'}{4c}(S^2-R^2)+\frac{1}{2}(\frac{h^2}{g}-\gamma_1)(R+S)-hJ,
\end{align}
\begin{align}\label{Rb}
R_t-c(\th)R_x=\frac{-c'}{4c}(S^2-R^2)+\frac{1}{2}(\frac{h^2}{g}-\gamma_1)(R+S)-hJ,
\end{align}
and the following balance laws,
\begin{align}
(S^2)_t+(c(\th)S^2)_x=\frac{c'}{2c}(S^2R-SR^2)+(\frac{h^2}{g}-\gamma_1)S(R+S)-2hSJ,
\end{align}
\begin{align}
(R^2)_t-(c(\th)R^2)_x=\frac{-c'}{2c}(S^2R-SR^2)+(\frac{h^2}{g}-\gamma_1)R(R+S)-2hRJ. 
\end{align}
So we can get the following equation
\begin{align}\label{div}
   \big(S^2+R^2\big)_t+\big(c(\th)(S^2-R^2)\big)_x=(\frac{h^2}{g}-\gamma_1)(R+S)^2-2hJ(R+S),
\end{align}
with the initial conditions
\begin{align}
R(x,0)=\ve \phi'(\frac{x}{\ve}),\,\,\,\,\,\,\, S(x,0)=(-2c(\theta_0(x))+\ve)\phi'(\frac{x}{\ve}).
\end{align}

By \eqref{max} we have
\begin{align}\nonumber
S(0,0)&=(-2c(\theta^*)+\epsilon)\phi'(0)\\\nonumber
&=(2c(\theta^*)-\epsilon)(-\phi'(0))\\
&>\max\left\{\frac{16C_U\|\gamma_1-\frac{h^2}{g}\|_{L^\infty}}{ c'(\theta*)\ln 2},\exp\big(\|\gamma_1-\frac{h^2}{g}\|_{L^\infty}\big)\right\}.\label{lbound}
\end{align}
% Now, by the decay of the energy we have
% \begin{equation}
% E(t)\leq 2\mathcal{E}(t)\leq 2\mathcal{E}(0)=\frac{1}{2}\int_{-\infty}^{\infty}(R^2+S^2+2u^2)(x,0)\,dx=O(\epsilon)
% \end{equation}
Since
\begin{align*}
    R^2(x,0)+S^2(x,0)=(-c(\th_0)+\epsilon)^2\phi'(\frac{x}{\epsilon})^2+c^2(\th_0)\phi'(\frac{x}{\epsilon})^2\leq C\,\phi'(\frac{x}{\epsilon})^2,
\end{align*}
for some constant $C$, it is easy to get that
\[\int_{\-\infty}^{\infty}R^2(x,0)+S^2(x,0)\,dx=O(\epsilon).\]
% \begin{align*}
%     \int_{-\infty}^{\infty}S^2(x,0)\,dx&=O(\epsilon)\\
%     \int_{-\infty}^{\infty}R^2(x,0)\,dx&=O(\epsilon^3)\\
%     \int_{-\infty}^{\infty}u^2(x,0)\,dx&=O(\epsilon^2).
% \end{align*}
Similarly,
\[\int_{-\infty}^{\infty}u^2(x,0)\,dx=O(\epsilon^2).\]

Under the help of energy decay we obtain
\begin{align}
    \mathcal{E}(t)\leq \mathcal{E}(0)=\frac{1}{2}\int_{-\infty}^{\infty}(R^2+S^2+2u^2)(x,0)\,dx=O(\epsilon).\label{4.16}
\end{align}

Hence, by \eqref{key0}, it is easy to get that 
\beq\label{est1}
\|J\|_{L^\infty((0,1),L^\infty(\R))}= O(\epsilon^\frac{1}{2}),
\eeq
where we use \eqref{key0}, \eqref{4.16} and 
%In fact,  the right hand side of \eqref{key0} is $O(\epsilon^\frac{1}{2})$ because first the energy $\mathcal{E}$ is of order $\epsilon$ and secondly by the design $J_0=u_0'+\frac{g}{h}\th_1$ is a function of order $\ve$ inside both $[-\epsilon,\epsilon]$ and $[\epsilon,\epsilon+2].$ Also, for $A_0$ we have the following estimate 
$$|A_0|\leq\int_\R|J_0|\,dx=\int_{-\epsilon}^{\epsilon}\epsilon\frac{h}{g}\phi'(x/\epsilon)\,dx+\int_\epsilon^{\epsilon+2}\chi'\,dx\approx \epsilon^2+\epsilon=O(\epsilon).$$
%NOTE: \|\th_0'(x)\|_{L^\infty}\leq C_2} see \eqref{max}.
Furthermore, $\epsilon$ is small enough such that \eqref{ce12} is satisfied.

%These estimates, along with estimate \eqref{Jest}, imply the following estimate for the quantity $J\equiv J,$
%\begin{equation}\label{Jest1}
%||J||_{L^\infty(\mathbb{R}\times (0,t))}=O(\sqrt{\epsilon}).
%\end{equation}
%Observe that $J^0:=J(x,0)$ is of order $\epsilon,$ that is
%$$\begin{align*}
    %\|J^0\|_{L^\infty(\mathbb{R})}=O(\epsilon).
%\end{align*}
\begin{figure}[h!]
\begin{center}
\includegraphics[width=10cm, height=10cm]{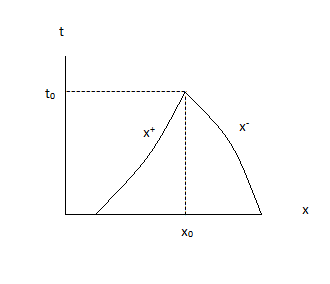})
\end{center}
\vspace{-2cm}
\caption{The triangle $\Omega.$}
\label{fig1}
\end{figure}

Now, we consider the two characteristic curves $x_{\pm}(t)$ given by
\begin{equation*}
\frac{dx_{\pm}}{dt}=\pm c(\theta)
\end{equation*}
with $x_1:=x_{+}(0)$ and $x_2:=x_{-}(0).$ See figure \ref{fig1}. Since the wave speed $c$ has a positive lower bound, two characteristics intersect at some point, say $(x_0,t_0),$ so we have $x_+(t_0)=x_-(t_0)=x_0$ and
\beq\label{4.17}|x_2-x_1|\leq|x_2-x_0|+|x_0-x_1|\leq 2C_Ut_0.\eeq

We assume that $t_0\leq 1$. We will verify it later by showing that blowup will happen before $t=1$.

Integrating \eqref{div} over the triangle $\Omega$ (see figure \ref{fig1}), and applying the divergence theorem, then we have
\begin{align}\nonumber
    \int_{x_0}^{x_1}2R^2(x,t_{+}(x))\,dx&+
    \int_{x_2}^{x_0}2S^2(x,t_{-}(x))\,dx-
    \int_{x_1}^{x_2}(R^2(x,0)+S^2(x,0))\,dx\\
&=\int\int_{\Omega}(\frac{h^2}{g}-\gamma_1)(S+R)^2\,dx\,dt
-\int\int_{\Omega}\frac{h}{g}\,J\,(S+R)\,dx\,dt.
\end{align}
Rearranging it, we have
\begin{align}\nonumber
    &\int_{x_1}^{x_0}R^2(x,t_{+}(x))\,dx+
    \int_{x_0}^{x_2}S^2(x,t_{-}(x))\,dx\\
    \leq& \frac{1}{2} \int_{x_1}^{x_2}(R^2+S^2)(x,0)\,dx\nonumber
+\iint_{\Omega}\frac{-1}{2}(\frac{h^2}{g}-\gamma_1)(S+R)^2
+h\,|J|\,(|S|+|R|)\,dx\,dt.\label{div1}
\end{align}
By the decay of energy and $t_0\leq 1$, we obtain 
\begin{align}
\iint_{\Omega}\frac{-1}{2}(\frac{h^2}{g}-\gamma_1)(S+R)^2\leq t_0\,\|\gamma_1-\frac{h^2}{g}\|_{L^\infty(\Omega)}\mathcal E(0)=O(\epsilon).
\end{align}
Similarly, by the decay of energy, $t_0\leq 1$, and also using \eqref{est1} and \eqref{4.17}, we have
\[
\iint_{\Omega}h\,|J|\,(|S|+|R|)\,dx\,dt\leq O(\epsilon).
\]
So we have 
\begin{align}\nonumber
 \int_{x_1}^{x_0}R^2(x,t_{+}(x))\,dx+
    \int_{x_0}^{x_2}S^2(x,t_{-}(x))\,dx
    \leq 
    O(\epsilon).
\end{align}

Next we consider a forward characteristic that we denote by $$x=\xi(t)$$ for $t\in[0,\min\{1, \frac{2\ln2}{\|\gamma_1-\frac{h^2}{g}\|_{L^\infty}}\}]$, such that $$\frac{d\xi(t)}{dt}=c(\theta(\xi(t),t)),$$ $$\xi(0)=0.$$  Integrating the equation $$\frac{d\theta(\xi(t),t)}{dt}= R(\xi(t),t),$$
 we obtain
\begin{align*}
|\theta(\xi(t),t)-\theta(\xi(0),0)|&=|\int_{0}^{t} R(\xi(s),s)\,ds|\\
&\leq \sqrt{t}\bigg(\int_{0}^{t}R^2(\xi(s),s)\,dt\bigg)^{1/2}\\
&=O(\sqrt{\epsilon}).
\end{align*}  
Using the smoothness of $c$, when $\epsilon$ is small enough, $$c'(\theta(\xi(t),t))>\frac{c'(\theta(\xi(0),0))}{2}=\frac{c'(\theta^*)}{2}>0.$$

Next, we claim that for smooth solutions we have $S(\xi(t),t)>1$ as long as $t\in\big[0,\min\{1,\frac{2\ln2}{\|\gamma_1-\frac{h^2}{g}\|_{L^\infty}}\}\big].$ To prove it by an contradiction argument, we assume that $S(\xi(t),t)\leq 1$ for some time in the interval. Define  \[t^*:=\inf\left\{t\in\big(0,\min\{1,\frac{2\ln2}{\|\gamma_1-\frac{h^2}{g}\|_{L^\infty}}\}\big]:S(t,\xi(t))\leq 1\right\}.\] By the continuity of $S$ we have
\[S(\xi(t^*),t^*)=1.\]

Define 
\begin{equation*}
\Tilde{S}=e^{p(t,x)}S(x,t):=\exp\big(\int_0^t\frac{1}{2}(\gamma_1-\frac{h^2}{g})(x,s)\,ds\big)S(x,t).
\end{equation*}
Some calculations give
\begin{align*}
    \Tilde{S}_t+c(\th)\Tilde{S}_x&=p_t(t,x)\Tilde{S}+e^{p(t,x)}S_t+c(\th)p_x(t,x)\Tilde{S}+c(\th)e^{p(t,x)}S_x\\
    &=(p_t+cp_x)\Tilde{S}+e^{p(t,x)}(S_t+cS_x)\\
    &=(p_t+cp_x)\Tilde{S}+e^{p(t,x)}\left[\frac{c'}{4c}(S^2-R^2)+\frac{1}{2}(\frac{h^2}{g}-\gamma_1)(R+S)-hJ\right]\\
    &=(p_t+cp_x)\Tilde{S}+\frac{c'}{4c}e^{-p(t,x)}\Tilde{S}^2+\frac{1}{2}(\frac{h^2}{g}-\gamma_1)\Tilde{S}-\frac{c'}{4c}e^{p(t,x)}R^2\\
    &+\frac{1}{2}(\frac{h^2}{g}-\gamma_1)e^{p(t,x)}R-e^{p(t,x)}hJ\\
    &=\big(p_t+cp_x+\frac{1}{2}(\frac{h^2}{g}-\gamma_1)\big)\Tilde{S}+\frac{c'}{4c}e^{-p(t,x)}\Tilde{S}^2-\frac{c'}{4c}e^{p(t,x)}R^2\\
    &+\frac{1}{2}(\frac{h^2}{g}-\gamma_1)e^{p(t,x)}R-e^{p(t,x)}hJ.
\end{align*}
This means
along the curve $\xi(t)$ for $t\in [0,t^*]$
\begin{align}\nonumber
\frac{d}{dt}\Tilde{S}(\xi(t),t)&=\big(\frac{d}{dt}p(t,\xi(t))
+\frac{1}{2}(\frac{h^2}{g}-\gamma_1)\big)\Tilde{S}
+\frac{c'}{4c}e^{-p(t,\xi(t))}\Tilde{S}^2-\frac{c'}{4c}e^{p(t,\xi(t))}R^2\\\nonumber
&+\frac{1}{2}(\frac{h^2}{g}-\gamma_1)e^{p(t,\xi(t))}R-e^{p(t,\xi(t))}hJ\\
&=\frac{c'}{4c}e^{-p(t,\xi(t))}\Tilde{S}^2-\frac{c'}{4c}e^{p(t,\xi(t))}R^2
+\frac{1}{2}(\frac{h^2}{g}-\gamma_1)e^{p(t,\xi(t))}R-e^{p(t,\xi(t))}hJ.
\end{align}
Dividing it by $\Tilde{S}^2$ then integrating it over $[0,t^*]$, we have
\begin{align*}
    \frac{1}{\Tilde{S}(0)}-\frac{1}{\Tilde{S}(t^*)}&\geq \int_0^{t^*}\frac{c'}{4c}e^{-p(t,\xi)}\,dt\\
    &+\int_0^{t^*}\frac{1}{\Tilde{S}^2}\left[-\frac{c'}{4c}e^{p(t,\xi)}R^2+\frac{1}{2}(\frac{h^2}{g}-\gamma_1)e^{p(t,\xi)}|R|-e^{p(t,\xi)}h|J|\right]\,dt.
\end{align*}
Recalling $0< t^*\leq \frac{2\ln 2}{\|\gamma_1-\frac{h^2}{g}\|_{L^\infty}},$ we have 
\begin{align}\nonumber
    \frac{1}{\Tilde{S}(t^*)}&\leq\frac{1}{\Tilde{S}(0)}-\int_0^{t^*}\frac{c'}{4c}e^{-p(t,\xi(t))}\,dt\\\nonumber
    &\quad+\int_0^{t^*}\frac{1}{\Tilde{S}^2}\left[\frac{c'}{4c}e^{p(t,\xi(t))}R^2+\frac{1}{2}(-\frac{h^2}{g}+\gamma_1)e^{p(t,\xi(t))}|R|
    +e^{p(t,\xi(t))}h|J|\right]\,dt\\\nonumber
    &\leq\min\left\{\frac{c'(\th^*)\ln 2}{16C_U\|\gamma_1-\frac{h^2}{g}\|_{L^\infty}},\exp\big(-\|\gamma_1-\frac{h^2}{g}\|_{L^\infty}\big),\,\frac{c'(\th^*)}{32C_U}\right\} -\frac{c'(\th^*)}{16C_U}t^*+D\sqrt{\epsilon}\\
    &\leq\exp\big(-\|\gamma_1-\frac{h^2}{g}\|_{L^\infty}\big)+D\sqrt{\epsilon},
\end{align}
for some positive constant $D$.

For $\ve$ small enough,
\[\frac{1}{\Tilde{S}(t^*)}<\exp\big(\frac{-1}{2}\|\gamma_1-\frac{h^2}{g}\|_{L^\infty}\big),\]
then
\[\Tilde{S}(t^*)>\exp\big(\frac{1}{2}\|\gamma_1-\frac{h^2}{g}\|_{L^\infty}\big).\]
Hence
\[S(t^*)>\exp\big(\frac{1}{2}\|\gamma_1-\frac{h^2}{g}\|_{L^\infty}\big)\exp\big(\int_0^{t^*}\frac{-1}{2}\big(\gamma_1-\frac{h^2}{g}\big)\,ds\big)\geq\exp\big(\frac{1}{2}\|\gamma_1-\frac{h^2}{g}\|_{L^\infty}(1-t^*)\big),\]
which gives $S(t^*)>1.$ This is a contradiction with the definition of $t^*.$\\

This proves that  $S(\xi(t),t)>1$ as long as $t\in\big[0,\min\{1,\frac{2\ln2}{\|\gamma_1-\frac{h^2}{g}\|_{L^\infty}}\}\big].$

Now using same calculations, we obtain
\[\frac{1}{\Tilde{S}(t)}\leq\min\left\{\frac{c'(\th^*)\ln 2}{16C_U\|\gamma_1-\frac{h^2}{g}\|_{L^\infty}},\frac{c'(\th^*)}{32C_U}\right\}-\frac{c'(\th^*)}{16C_U}t+D\sqrt{\epsilon}.\]
As a consequence,   $\tilde S(t)$ blows up before
\[t=\min\left\{\frac{\ln 2}{\|\gamma_1-\frac{h^2}{g}\|_{L^\infty}},\frac{1}{2}\right\}
+\frac{16C_UD}{c'(\th^*)}\sqrt{\epsilon}.\]

Now choose $\epsilon$ small enough, we know the solution will blowup before $$T=\min\left\{\frac{2\ln2}{\|\gamma_1-\frac{h^2}{g}\|_{L^\infty}},1\right\}.$$ More precisely, there exists a time $t_0<T$ such that \[S(t,\xi(t))\to\infty\] as ${t\to t_0^{-}}.$ 
This shows that $\th_t\to \infty$ or $\th_x\to -\infty$ as ${t\to t_0^{-}}.$

On the other hand, because of the smallness of $R$ initially, i.e. $R(x,0)$ is of order $O(\epsilon),$ and \eqref{Rb}, we know that $R(x,t)$ remains uniformly bounded before the blowup of $S(x,t).$ This shows that both $$\th_t\to \infty,\hbox{ and \ } \th_x\to -\infty$$ simultaneously as ${t\to t_0^{-}}$, at the blowup point.

This completes the proof of Theorem \ref{thm}.

% However, similar calculations show that $R$ remains bounded. Hence, we must have both limits hold true
% $\th_t\to \infty$ and $\th_x\to -\infty$ as ${t\to t_0^{-}}.$

\section*{Acknowledgments}
The authors are partially supported by NSF grant DMS-2008504. This paper is motivated by a discussion with Weishi Liu. The authors thank Weishi Liu for the helpful comments.

\section*{Conflict of interest statement}

On behalf of all authors, the corresponding author states that there is no conflict of interest.

\end{document}